\documentclass[article,12pt]{amsart}
\usepackage{lineno,hyperref,enumitem}
\usepackage{amsmath,amssymb,dsfont}

\newtheorem{thm}{Theorem}[section]

\newtheorem{prop}[thm]{Proposition}

\newtheorem{lem}[thm]{Lemma}
\newtheorem{rem}[thm]{Remark}

\newtheorem{ex}[thm]{Example}

\newcommand{\ds}{\displaystyle}

\begin{document}

\title[A note on weaving fusion frames]{A note on weaving fusion frames}


\author[A. Bhandari]{Animesh Bhandari}
\address{Department of Mathematics\\ SRM University, {\it AP} - Andhra Pradesh\\ Neeru Konda, Amaravati, Andhra Pradesh - 522502\\ India}
\email{bhandari.animesh@gmail.com, animesh.b@srmap.edu.in}

\begin{abstract}
Fusion frames are widely studied for their applications in recovering signals from large data. These are proved to be very useful in many areas, such as, distributed processing, wireless sensor networks, packet encoding. Inspired by the work of Bemrose et al.\cite{Be16}, this paper delves into the properties and characterizations of weaving fusion frames.


\end{abstract}

\maketitle

\section{Introduction}
The concept of Hilbert space frames was first introduced by Duffin and Schaeffer \cite{DuSc52} in 1952. After several decades, in 1986, the importance of frame
theory was popularized by work as in the groundbreaking work by Daubechies, Grossman and Meyer \cite{DaGrMa86}. Since then frame theory has been widely used by mathematicians and engineers in various fields of mathematics and engineering, namely, operator theory \cite{HaLa00}, harmonic analysis \cite{Gr01}, wavelet analysis \cite{GuSh18}, signal processing \cite{Fe99}, image processing \cite{ArZe16}, sensor network \cite{CaKuLiRo07}, data analysis \cite{CaKu12}, Retro Banach Frame \cite{VeVaSi19}, etc.

Frame theory literature became richer through several generalizations-fusion frame (frames of subspaces) \cite{CaKu04, CaKuLi08} , $G$-frame (generalized frames) \cite{Su06}, $K$-frame (atomic systems) \cite{Ga12},
$K$-fusion frame (atomic subspaces) \cite{Bh17},
etc. and these generalizations have been proved to be useful in various applications.

Classical frames have been instrumental in signal processing and functional analysis, providing a stable and redundant way to represent signals. However, they face significant limitations in distributed processing, particularly when it comes to projecting signals onto multidimensional subspaces. This limitation is crucial in applications like wireless sensor networks, where data is collected and processed across multiple sensors, and in packet encoding, where robustness and redundancy are essential. To address these challenges, fusion frames were introduced, extending the concept of frames to collections of subspaces with associated weights.

Fusion frames have proven to be highly effective in distributed processing, enabling more flexible and stable signal representations across multiple subspaces. This has made them particularly valuable in practical applications like wireless sensor networks, distributed signal processing, and error-resilient data transmission in packet encoding.

Beyond these practical uses, fusion frames have also emerged as a powerful tool in theoretical research. They play a significant role in the solution of the Kadison-Singer problem, a long-standing question in operator theory, and in optimal subspace packing, which is important for coding theory and communications. The rapid development of fusion frame theory over the years has led to a wide range of applications and a deeper understanding of their mathematical properties.

In this paper, we explore the various properties and characterizations of weaving fusion frames, a concept that extends traditional fusion frames by allowing the interweaving of subspaces. We delve into their structural aspects and provide theoretical insights into their stability and robustness. 


Throughout this paper, $\mathcal{H}$ is a separable Hilbert space. We denote by $\mathcal{L}(\mathcal{H}_1, \mathcal{H}_2)$ the space of all bounded linear operators from  $\mathcal{H}_1$ into $\mathcal{H}_2$, and $\mathcal L(\mathcal H)$ for $\mathcal L(\mathcal H, \mathcal H)$. For $T \in \mathcal{L}\mathcal{(H)}$, we denote $D(T), N(T)$ and $R(T)$ for domain, null space and range of $T$, respectively. For a collection of closed subspaces $\mathcal W_i$ of $\mathcal H$ and scalars $w_i$, $i\in I$, the weighted collection of closed subspaces $\lbrace (\mathcal W_i , w_i) \rbrace_{i \in I}$ is denoted by $\mathcal W_w$. We consider $I$ to be countable index set, $\mathcal I$ is the identity operator and $P_{\mathcal V}$ is the orthogonal projection onto $\mathcal V$.

\section{Preliminaries}\label{Sec-Preli}

Before diving into the main sections, throughout this section we recall basic definitions and results needed in this paper. For detailed discussion regarding frames and its applications we refer \cite{CaKu12, Ch08}.

\subsection{Frame} A collection  $\{ f_i \}_{i\in I}$ in $\mathcal{H}$ is called a \emph{frame} if there exist constants $A,B >0$ such that \begin{equation}\label{Eq:Frame} A\|f\|^2~ \leq ~\sum_{i\in I} |\langle f,f_i\rangle|^2 ~\leq ~B\|f\|^2,\end{equation}for all $f \in \mathcal{H}$. The numbers $A, B$ are called \emph{frame bounds}. The supremum over all $A$'s and infimum over all $B$'s satisfying above inequality are called the \emph{optimal frame bounds}.
If a collection satisfies only the right inequality in (\ref{Eq:Frame}), it is called a {\it Bessel sequence}.

Given a frame $\{f_i\}_{i\in I}$ for $\mathcal{H}$, the \emph{pre-frame operator} or \emph{synthesis operator} is a bounded linear operator $T: l^2(I)\rightarrow \mathcal{H}$ and is defined by $T\{c_i\}_{i \in I} = \ds \sum_{i\in I} c_i f_i$. The adjoint of $T$, $T^*: \mathcal{H} \rightarrow l^2(I)$, given by $T^*f = \{\langle f, f_i\rangle\}_{i \in I}$, is called the \emph{analysis operator}. The \emph{frame operator}, $S=TT^*: \mathcal{H}\rightarrow \mathcal{H}$, is defined by $$Sf=TT^*f = \sum_{i\in I} \langle f, f_i\rangle f_i.$$ It is well-known that the frame operator is bounded, positive, self adjoint and invertible.

\subsection{Fusion frame}\label{Fusion} Consider a weighted collection of closed subspaces, $\mathcal W_w=\{(\mathcal W_i, w_i)\}_{i \in I}$, of $\mathcal H$. Then $\mathcal W_w$  is said to be a fusion frame for $\mathcal H$, if there exist constants $0<A \leq B <\infty$ satisfying
\begin{equation}\label{Eq:Fus-frame}
	A\|f\|^2 \leq \sum_{i \in I} w_i ^2 \|P_{\mathcal W_i}f\|^2 \leq B\|f\|^2, \end{equation}
where $P_{\mathcal W_i}$ is the orthogonal projection from $\mathcal H$ onto $\mathcal W_i$.
The constants $A$ and $B$ are called {\it fusion frame bounds}.
A collection of closed subspaces, satisfying only the right inequality in (\ref{Eq:Fus-frame}), is called a fusion Bessel sequence.


For a family of closed subspaces, $\lbrace \mathcal W_i \rbrace_{i \in I}$, of $\mathcal H$, the associated $l^2$ space is defined by $\left(\sum\limits_{i \in I} \bigoplus \mathcal W_i\right)_{l^2} = \lbrace \lbrace f_i \rbrace_{i \in I} : f_i \in \mathcal W_i, \sum\limits_{i \in I} \|f_i\|^2 < \infty \rbrace$ with the inner product $\left \langle \lbrace f_i \rbrace_{i \in I}, \lbrace g_i \rbrace_{i \in I} \right \rangle_{\left(\sum\limits_{i \in I} \bigoplus \mathcal W_i\right)_{l^2}} = \sum\limits_{i \in I} \langle f_i , g_i \rangle_{\mathcal{H}} $ and the norm is $\|\{f_i\}_{i \in I}\|_{\left(\sum\limits_{i \in I} \bigoplus \mathcal W_i\right)_{l^2}}^2 = \sum\limits_{i \in I}\|f_i\|^2$. It is easy to see that $\left(\sum\limits_{i \in I} \bigoplus \mathcal W_i\right)_{l^2}$ is a Hilbert space. In this context the corresponding dense inclusion is defined as $\mathcal L_{\mathcal W}^{00}=\left(\sum\limits_{i \in I} \bigoplus \mathcal W_i\right)_{l^{00}}=\{\{f_i\}_{i \in I}\in \{\mathcal W_i\}_{i \in I}: f_i=0 \text{~~for~~ all ~~but ~~finitely ~~many~~} i\} \subseteq \left(\sum\limits_{i \in I} \bigoplus \mathcal W_i\right)_{l^2} = \mathcal L_{\mathcal W}^2$. For the analogous dense inclusion we have $\mathcal L_{\mathcal H}^{00} = \left(\sum\limits_{i \in I} \bigoplus \mathcal H_i\right)_{l^{00}} \subseteq \left(\sum\limits_{i \in I} \bigoplus \mathcal H_i\right)_{l^2} = \mathcal L_{\mathcal H}^2$.

Let $\mathcal W_w$ be a fusion frame. Then the associated synthesis operator,
$T_{\mathcal W}: D(T_{\mathcal W}) \subseteq \mathcal L_{\mathcal H}^2\rightarrow \mathcal H$ is defined as $T_{\mathcal W} (\lbrace f_i \rbrace_{i \in I})=\sum \limits_{i \in I} w_i P_{\mathcal W_i}f_i$ , where $D(T_{\mathcal W})=\left\{\{f_i\}_{i \in I} \in \mathcal L_{\mathcal H}^2: \sum\limits_{i \in I}w_i^2P_{\mathcal W_i}f_i  ~~\text{convergent}\right \}$. Since $\mathcal L_{\mathcal H}^{00} \subset D(T_{\mathcal W})$ and it is dense in $\mathcal L_{\mathcal H}^2$, hence the synthesis operator is densely defined. 

On the other hand for every $f \in \mathcal H$ and $\{f_i\}_{i \in I} \in \mathcal L_{\mathcal H}^2$ we obtain,
\begin{eqnarray*}
\langle T_{\mathcal W}^*f, \{f_i\}_{i \in I} \rangle_{\mathcal L_{\mathcal H}^2} = \langle f,  T_{\mathcal W} (\{f_i\}_{i \in I})\rangle_{\mathcal H} &=& \left \langle f,  \sum \limits_{i \in I} w_i P_{\mathcal W_i}f_i \right\rangle_{\mathcal H}\\
&=& \sum \limits_{i \in I} \langle f, w_i P_{\mathcal W_i} f_i \rangle\\
&=& \sum \limits_{i \in I} \langle w_i P_{\mathcal W_i} f,  f_i \rangle\\
&=& \langle \{w_i P_{\mathcal W_i} f\}_{i \in I}, \{f_i\}_{i \in I} \rangle
\end{eqnarray*}

and hence the adjoint of synthesis operator,  $T^*_{\mathcal W}: D(T^*_{\mathcal W}) \subseteq \mathcal H \rightarrow  \mathcal L_{\mathcal H}^2 $ is defined as  $T^*_{\mathcal W} (f)=\lbrace v_i P_{\mathcal W_i} (f) \rbrace_{i \in I} $, which is known as analysis operator, where $D(T^*_{\mathcal W})=\{f \in \mathcal H: \{w_i P_{\mathcal W_i} f\}_{i \in I} \in \mathcal L_{\mathcal H}^2\}$. It is well-known that (see \cite{CaKu04}) the synthesis operator $ T_{\mathcal W}$ of a fusion frame is bounded, linear and onto, whereas the corresponding analysis operator $ T_{\mathcal W} ^*$  is (possibly into) an isomorphism. If we consider the composition of synthesis and analysis operator we obtain the corresponding fusion frame operator which is defined as $S_{\mathcal W} (f) = T_{\mathcal W}T_{\mathcal W}^*(f)$ $=\sum \limits_{i \in I} w_i ^2 P_{\mathcal W_i}(f)$. $S_\mathcal{W}$ is bounded, positive, self adjoint and invertible. Thus  every $f \in \mathcal H$ can be expressed by its fusion frame measurements $\lbrace w_i P_{\mathcal W_i} f \rbrace_{i \in I}$ as
\begin{equation}\label{Eq:Fusion-frame-recons} f=\sum_{i \in I} v_i S_\mathcal W ^{-1} (v_i P_{\mathcal W_i} f) .\end{equation}


\subsection{Weaving Fusion frames}
Let   $\mathcal V_v=\{(\mathcal V_i, v_i)\}_{i \in I}$ and  $\mathcal W_w=\{(\mathcal W_i, w_i)\}_{i \in I}$ be two fusion frames for $\mathcal H$. Then they are called weaving fusion frames for $\mathcal H$ if for every $\sigma \subset I$ and every $f \in \mathcal H$ there exist finite positive constants $A \leq B$ so that $\{(\mathcal V_i, v_i)\}_{i \in \sigma} \cup \{(\mathcal W_i, w_i)\}_{i \in \sigma^c}$ is a fusion frame for $\mathcal H$ with the universal bounds $A \leq B$, i.e.  the following inequality is satisfied:
\begin{equation}\label{weaving}
A \|f\|^2 \leq \sum_{i \in \sigma} v_i ^2 \|P_{\mathcal V_i}f\|^2 +  \sum_{i \in \sigma^c} w_i ^2 \|P_{\mathcal W_i}f\|^2\leq B\|f\|^2.
\end{equation}

\begin{ex}
Let us consider an orthonormal basis $\{e_n\}_{n=1}^\infty$ for $\mathcal H$. Suppose for every $n$, $\mathcal V_n = \text{span} \{e_n\}$ and $\mathcal W_n = \text{span} \{e_n, e_{n+1}\}$. Since for every $f \in \mathcal H$ and $\sigma \subset \{1, 2, \cdots\}$ we have,
$$\|f\|^2 \leq \sum_{n \in \sigma}  \|P_{\mathcal V_n}f\|^2 +  \sum_{n \in \sigma^c}  \|P_{\mathcal W_n}f\|^2\leq 2\|f\|^2.$$

Therefore, $\{(\mathcal V_n, 1)\}_{n=1}^\infty$ and $\{(\mathcal W_n, 1)\}_{n=1}^\infty$ are weaving fusion frames for $\mathcal H$.
\end{ex}

In the following example we discuss a non-example of weaving fusion frames.

\begin{ex}
Let us consider an orthonormal basis $\{e_n\}_{n=1}^\infty$ for $\mathcal H$. Suppose for every $n$, $\mathcal V_n = \text{span} \{e_n\}$ and $\mathcal W_1 = \text{span} \{e_2\}$, $\mathcal W_2 = \text{span} \{e_1\}$ and $\mathcal W_n = \text{span} \{e_n\}$ for $n \geq 3$. Since for  $\sigma =\{2\} \subset \{1, 2, \cdots\}$ we have,
$$\sum_{n \in \sigma}  \|P_{\mathcal W_n}e_2\|^2 +  \sum_{n \in \sigma^c}  \|P_{\mathcal V_n}e_2\|^2=0$$

Thus, $\{(\mathcal V_n, 1)\}_{n=1}^\infty$ and $\{(\mathcal W_n, 1)\}_{n=1}^\infty$ are not weaving fusion frames for $\mathcal H$.
\end{ex}

\begin{rem}\label{Remark}
For weaving fusion frames we can define the associated weaving synthesis, analysis and weaving fusion frame operators analogous to the section \ref{Fusion}.
\end{rem}

For detailed discussion regarding weaving fusion frames we refer \cite{De17}.

\section{Main Results}\label{Sec-char}

In this section we discuss various characterizations of weaving fusion frames.

\begin{thm}
Let $\mathcal V_v=\{(\mathcal V_i, v_i)\}_{i \in I}$ and  $\mathcal W_w=\{(\mathcal W_i, w_i)\}_{i \in I}$ be two families of weighted closed subspaces of $\mathcal H$. Suppose $\{f_{ij}\}_{j \in J_i}$ and $\{g_{ij}\}_{j \in J_i}$ are frames for $\mathcal V_i$ and $\mathcal W_i$ with bounds $A_i, B_i$ and $C_i, D_i$ respectively for every $i \in I$. If $0<A=\inf\limits_{i \in I} A_i \leq \sup\limits_{i \in I}B_i =B <\infty$ and $0<C=\inf\limits_{i \in I} C_i \leq \sup\limits_{i \in I}D_i =D <\infty$. Then the following are equivalent:
\begin{enumerate}
\item $\{(\mathcal V_i, v_i)\}_{i \in I}$ and $\{(\mathcal W_i, w_i)\}_{i \in I}$ are weaving fusion frames for $\mathcal H$.

\item $\{v_i f_{ij}\}_{i \in I, j \in J_i}$ and $\{w_i g_{ij}\}_{i \in I, j \in J_i}$ are weaving frames for $\mathcal H$.
\end{enumerate}
\end{thm}

\begin{proof}
Let us suppose $\alpha = \text{min} (A, C)$ and $\beta = \text{max} (B, D)$.

\noindent (\underline{1 $\implies$ 2}) Since $\{f_{ij}\}_{j \in J_i}$ and $\{g_{ij}\}_{j \in J_i}$ are frames for $\mathcal V_i$ and $\mathcal W_i$ respectively with the respective bounds, then for every $f \in\mathcal H$ and every $\sigma \subset I$ we have,
\begin{eqnarray*}
 \alpha  ( \sum\limits_{i \in \sigma} v_i^2 \|P_{\mathcal V_i}f\|^2  &+ & \sum\limits_{i \in \sigma^c} w_i^2 \|P_{\mathcal W_i}f\|^2  ) \\
& \leq & A \sum\limits_{i \in \sigma} v_i^2 \|P_{\mathcal V_i}f\|^2 + C \sum\limits_{i \in \sigma^c} w_i^2 \|P_{\mathcal W_i}f\|^2\\
&\leq & \sum\limits_{i \in \sigma} v_i^2 A_i \|P_{\mathcal V_i}f\|^2 +  \sum\limits_{i \in \sigma^c} w_i^2 C_i \|P_{\mathcal W_i}f\|^2\\
& \leq & \sum\limits_{i \in \sigma} v_i^2 \sum\limits_{j \in J_i} |\langle P_{\mathcal V_i}f, f_{ij} \rangle|^2 +  \sum\limits_{i \in \sigma} w_i^2 \sum\limits_{j \in J_i} |\langle P_{\mathcal W_i}f, g_{ij} \rangle|^2\\
& = & \sum\limits_{i \in \sigma} \sum\limits_{j \in J_i} |\langle f, v_i f_{ij} \rangle|^2 +  \sum\limits_{i \in \sigma}  \sum\limits_{j \in J_i} |\langle f, w_i g_{ij} \rangle|^2\\
& \leq & \sum\limits_{i \in \sigma} v_i^2 B_i \|P_{\mathcal V_i}f\|^2 +  \sum\limits_{i \in \sigma^c} w_i^2 D_i \|P_{\mathcal W_i}f\|^2\\
& \leq & B \sum\limits_{i \in \sigma} v_i^2 \|P_{\mathcal V_i}f\|^2 +  D \sum\limits_{i \in \sigma^c} w_i^2  \|P_{\mathcal W_i}f\|^2\\
& \leq & \beta (\sum\limits_{i \in \sigma} v_i^2 \|P_{\mathcal V_i}f\|^2 +   \sum\limits_{i \in \sigma^c} w_i^2  \|P_{\mathcal W_i}f\|^2)
\end{eqnarray*}
Thus if $\{(\mathcal V_i, v_i)\}_{i \in I}$ and $\{(\mathcal W_i, w_i)\}_{i \in I}$ are weaving fusion frames for $\mathcal H$ with bounds $A_{\mathcal V \mathcal W} \leq B_{\mathcal V \mathcal W}$, then for every $f \in \mathcal H$ and every $\sigma \subset I$ we have,
$$\alpha A_{\mathcal V \mathcal W} \|f\|^2 \leq \sum\limits_{i \in \sigma}  \sum\limits_{j \in J_i} |\langle f, v_i f_{ij} \rangle|^2 + \sum\limits_{i \in \sigma^c}  \sum\limits_{j \in J_i} |\langle f, w_i f_{ij} \rangle|^2 \leq  \beta B_{\mathcal V \mathcal W}.$$
Therefore, $\{v_i f_{ij}\}_{i \in I, j \in J_i}$ and $\{w_i g_{ij}\}_{i \in I, j \in J_i}$ are weaving frames for $\mathcal H$.\\

\noindent (\underline{2 $\implies$ 1}) Conversely, if $\{v_i f_{ij}\}_{i \in I, j \in J_i}$ and $\{w_i g_{ij}\}_{i \in I, j \in J_i}$ are weaving frames for $\mathcal H$ with bounds $A_{vw} \leq B_{vw}$, then for every $f \in \mathcal H$ and every $\sigma \subset I$ we have,
$$\frac{A_{vw} } {\beta} \leq \sum\limits_{i \in \sigma} v_i^2 \|P_{\mathcal V_i}f\|^2 +   \sum\limits_{i \in \sigma^c} w_i^2  \|P_{\mathcal W_i}f\|^2 \leq \frac{B_{vw} } {\alpha}.$$
Consequently, $\{(\mathcal V_i, v_i)\}_{i \in I}$ and $\{(\mathcal W_i, w_i)\}_{i \in I}$ are weaving fusion frames for $\mathcal H$.
\end{proof} 

We characterize weaving fusion frames by means of the associated weaving fusion frame operator.

\begin{lem}
Let $\{(\mathcal V_i, v_i)\}_{i \in I}$ and  $\{(\mathcal W_i, w_i)\}_{i \in I}$ be two fusion frames for $\mathcal H$ with bounds $A \leq B$ and $C \leq D$ respectively. Then the following are equivalent:
\begin{enumerate}
\item $\{(\mathcal V_i, v_i)\}_{i \in I}$ and $\{(\mathcal W_i, w_i)\}_{i \in I}$ are weaving fusion frames for $\mathcal H$.

\item For every $\sigma \subset I$, suppose $S_{\mathcal V \mathcal W_{\sigma}}$ is the corresponding weaving fusion frame operator, then for every $f \in \mathcal H$ there exist $\alpha >0$ independent of $\sigma$, we have $\|S_{\mathcal V \mathcal W_{\sigma}} f\| \geq \alpha \|f\|.$
\end{enumerate}
\end{lem}

\begin{proof}
\noindent (\underline{1 $\implies$ 2}) If $\{(\mathcal V_i, v_i)\}_{i \in I}$ and $\{(\mathcal W_i, w_i)\}_{i \in I}$ are weaving fusion frames for $\mathcal H$ with the universal bounds $\alpha \leq \beta$, then for every $\sigma \subset I$ and every $f \in \mathcal H$ we have,
\begin{equation}\label{weaving eqn}
\alpha \|f\|^2 \leq \sum\limits_{i \in \sigma} v_i^2 \|P_{\mathcal V_i}f\|^2 +   \sum\limits_{i \in \sigma^c} w_i^2  \|P_{\mathcal W_i}f\|^2 \leq \beta \|f\|^2.
\end{equation}
Suppose $S_{\mathcal V \mathcal W_{\sigma}}$ is the fusion frame operator for the associated weaving, then using inequality (\ref{weaving eqn}) we have,
$\alpha \|f\|^2 \leq \langle S_{\mathcal V \mathcal W_{\sigma}} f, f \rangle \leq \beta \|f\|^2.$

Thus, $\|S_{\mathcal V \mathcal W_{\sigma}}f \| = \sup\limits_{\|g\|=1} |\langle S_{\mathcal V \mathcal W_{\sigma}}f, g \rangle| \geq \left \langle S_{\mathcal V \mathcal W_{\sigma}}f, \frac {f} {\|f\|} \right \rangle \geq \alpha \|f\|.$\\

\noindent (\underline{2 $\implies$ 1}) Applying Remark (\ref{Remark}), let us suppose for every $\sigma \subset I$, $T_{\mathcal V \mathcal W_{\sigma}}$ and $T_{\mathcal V \mathcal W_{\sigma}}^*$ be the associated weaving synthesis and analysis operators respectively. Then for every $f \in \mathcal H$ we have,

$\alpha^2 \|f\|^2 \leq \|S_{\mathcal V \mathcal W_{\sigma}}f \|^2=\|T_{\mathcal V \mathcal W_{\sigma}} T_{\mathcal V \mathcal W_{\sigma}}^* f\|^2 \leq \|T_{\mathcal V \mathcal W_{\sigma}}\|^2 \|T_{\mathcal V \mathcal W_{\sigma}}^* f\|^2$ and hence we obtain,

$\sum\limits_{i \in \sigma} v_i^2 \|P_{\mathcal V_i}f\|^2 +   \sum\limits_{i \in \sigma^c} w_i^2  \|P_{\mathcal W_i}f\|^2 =\|T_{\mathcal V \mathcal W_{\sigma}}^* f\|^2 \geq \frac {\alpha^2} {B^2+D^2} \|f\|^2.$

On the other hand the universal upper bound for the corresponding weaving will be obtained from the [Proposition 3.1, \cite{Be16}].
\end{proof}

\begin{prop}
Let $\mathcal V_v=\{(\mathcal V_i, v_i)\}_{i \in I}$ and $\mathcal W_w=\{(\mathcal W_i, w_i)\}_{i \in I}$ be two weighted sequences of closed subspaces of $\mathcal H$. Then the following are equivalent:

\begin{enumerate}
\item $\mathcal V_v$ and $\mathcal W_w$ are weaving fusion Bessel sequences for $\mathcal H$.

\item For every $\sigma \subset I$, the corresponding weaving synthesis operator $T_{\mathcal V \mathcal W_{\sigma}}: \mathcal L_{\mathcal V \mathcal W}^2=\left(\sum\limits_{i \in \sigma} \bigoplus \mathcal V_i + \sum\limits_{i \in \sigma^c} \bigoplus \mathcal W_i\right)_{l^2} \rightarrow \mathcal H$ is bounded.

\item For every $\sigma \subset I$, the associated weaving analysis operator $T_{\mathcal V \mathcal W_{\sigma}}^*: \mathcal H \rightarrow \mathcal L_{\mathcal V \mathcal W}^2$ is bounded.

\item For every $\sigma \subset I$, the corresponding weaving fusion frame operator $S_{\mathcal V \mathcal W_{\sigma}} : \mathcal H \rightarrow \mathcal H$ is bounded.
\end{enumerate}
\end{prop}

\begin{proof}
\noindent (\underline{1 $\implies$ 2}) Let us suppose $\mathcal V_v$ and $\mathcal W_w$ are weaving fusion Bessel sequences for $\mathcal H$ with bound $B_{\mathcal V \mathcal W}$. Then for every $\sigma \subset I$ and every $f \in \mathcal H$ we have,
$ \sum\limits_{i \in \sigma} v_i^2 \|P_{\mathcal V_i}f\|^2 +   \sum\limits_{i \in \sigma^c} w_i^2  \|P_{\mathcal W_i}f\|^2 \leq B_{\mathcal V \mathcal W} \|f\|^2.$
Thus for any $\{f_i\}_{i \in \sigma} \cup \{g_i\}_{i \in \sigma^c} \in \mathcal L_{\mathcal V \mathcal W}^2$ we obtain,
\begin{eqnarray*}
\left\|\sum\limits_{i \in \sigma} v_i P_{\mathcal V_i}f_i + \sum \limits_{i \in \sigma^c} w_i P_{\mathcal W_i}g_i \right\| &=& \sup\limits_{\|g\|=1} \left | \left \langle \sum\limits_{i \in \sigma} v_i P_{\mathcal V_i}f_i + \sum \limits_{i \in \sigma^c} w_i P_{\mathcal W_i}g_i , g \right \rangle \right| \\
&=& \sup\limits_{\|g\|=1}  \left | \left \langle \sum\limits_{i \in \sigma} v_i P_{\mathcal V_i}f_i, g \right \rangle +  \left \langle \sum \limits_{i \in \sigma^c} w_i P_{\mathcal W_i}g_i , g \right \rangle\right| \\
&=& \sup\limits_{\|g\|=1}  \left | \sum\limits_{i \in \sigma} v_i \langle f_i,  P_{\mathcal V_i} g \rangle  + \sum\limits_{i \in \sigma^c} w_i \langle g_i,  P_{\mathcal W_i} g \rangle\right| \\
& \leq &  \sup\limits_{\|g\|=1} \left ( \sum\limits_{i \in \sigma} v_i ^2 \| P_{\mathcal V_i} g\|^2 \|f_i\|^2  \right )^{\frac{1}{2}} \\
& + & \sup\limits_{\|g\|=1} \left ( \sum\limits_{i \in \sigma^c} w_i ^2 \| P_{\mathcal W_i} g\|^2 \|g_i\|^2  \right )^{\frac{1}{2}} \\
& \leq &  \sup\limits_{\|g\|=1} \left ( \sum\limits_{i \in \sigma} v_i ^2 \| P_{\mathcal V_i} g\|^2 \right )^{\frac{1}{2}} \left ( \sum\limits_{i \in \sigma}  \|f_i\|^2 \right )^{\frac{1}{2}} \\
&+&  \sup\limits_{\|g\|=1} \left ( \sum\limits_{i \in \sigma^c} w_i ^2 \| P_{\mathcal W_i} g\|^2 \right )^{\frac{1}{2}} \left ( \sum\limits_{i \in \sigma^c}  \|g_i\|^2 \right )^{\frac{1}{2}} \\
& \leq & 2 \sqrt{B_{\mathcal V \mathcal W}} \|f\|
\end{eqnarray*}
Therefore, $T_{\mathcal V \mathcal W_{\sigma}}$ is bounded.\\

Analogously (\underline{2 $\implies$ 3}) and (\underline{3 $\implies$ 4}) will be satisfied.\\

\noindent (\underline{4 $\implies$ 1}) Suppose  $S_{\mathcal V \mathcal W_{\sigma}}$ is bounded by the bound $B_{\mathcal V \mathcal W}$, then for $\sigma \subset I$ and every $f \in \mathcal H$ we have, 
\begin{eqnarray*}
\sum\limits_{i \in \sigma} v_i^2 \|P_{\mathcal V_i}f\|^2 + \sum \limits_{i \in \sigma^c} w_i^2 \|P_{\mathcal W_i}f\|^2 = \langle S_{\mathcal V \mathcal W_{\sigma}} f, f  \rangle & \leq & \|S_{\mathcal V \mathcal W_{\sigma}} f\| \|f\|\\
&\leq & B_{\mathcal V \mathcal W} \|f\|^2.
\end{eqnarray*}
Hence $\mathcal V_v$ and $\mathcal W_w$ are weaving fusion Bessel sequences.
\end{proof}

\begin{thm}\label{1}
Let $\mathcal V_v=\{(\mathcal V_i, v_i)\}_{i \in I}$ and $\mathcal W_w=\{(\mathcal W_i, w_i)\}_{i \in I}$ be two weighted sequences of closed subspaces of $\mathcal H$. Then the following are equivalent:

\begin{enumerate}
\item $\mathcal V_v$ and $\mathcal W_w$ are weaving fusion frames for $\mathcal H$.

\item For every $\sigma \subset I$, the corresponding weaving synthesis operator $T_{\mathcal V \mathcal W_{\sigma}}: \mathcal L_{\mathcal V \mathcal W}^2=\left(\sum\limits_{i \in \sigma} \bigoplus \mathcal V_i + \sum\limits_{i \in \sigma^c} \bigoplus \mathcal W_i\right)_{l^2} \rightarrow \mathcal H$ is bounded, surjective operator.

\item For every $\sigma \subset I$, the associated weaving analysis operator $T_{\mathcal V \mathcal W_{\sigma}}^*: \mathcal H \rightarrow \mathcal L_{\mathcal V \mathcal W}^2$ is bounded, injective operator and has closed range.
\end{enumerate}
\end{thm}

\begin{proof}
The proof is easily followed from \cite{CaKuLi08}.
\end{proof}

\begin{lem}\label{Lemma}
For every $\sigma \subset I$, let us consider the associated Hilbert direct sum $\left(\sum\limits_{i \in \sigma} \bigoplus \mathcal V_i + \sum\limits_{i \in \sigma^c} \bigoplus \mathcal W_i\right)_{l^2}$ of the closed subspaces $\{\mathcal V_i \}_{i \in I}$ and $\{\mathcal W_i \}_{i \in I}$. Suppose $\mathcal U_i \subset \mathcal V_i$ and $\mathcal X_i \subset \mathcal W_i$ for every $i \in I$. Then 
$\left(\sum\limits_{i \in \sigma} \bigoplus \mathcal U_i + \sum\limits_{i \in \sigma^c} \bigoplus \mathcal X_i\right)_{l^2} ^{\perp} = \left(\sum\limits_{i \in \sigma} \bigoplus \mathcal U_i^{\perp} + \sum\limits_{i \in \sigma^c} \bigoplus \mathcal X_i^{\perp} \right)_{l^2} $
\end{lem}

\begin{proof}
The proof is followed from \cite{Ko23}.
\end{proof}

Let us define weaving fusion Riesz bases analogous to fusion Riesz basis \cite{CaKu04, Sh20}.

Suppose $\mathcal V_v=\{(\mathcal V_i, v_i)\}_{i \in I}$ and $\mathcal W_w=\{(\mathcal W_i, w_i)\}_{i \in I}$ be two weighted sequences of closed subspaces of $\mathcal H$. Then they are said to be weaving fusion Riesz bases if for every $\sigma \subset I$ and every $\{f_i\}_{i \in \sigma} \cup \{g_i\}_{i \in \sigma^c} \in \mathcal L_{\mathcal V \mathcal W}^{00}$ there are finite positive constants $A_{\mathcal V \mathcal W} \leq B_{\mathcal V \mathcal W}$ we have,

\begin{equation}\label{fusion Riesz}
A_{\mathcal V \mathcal W} \left (\sum\limits_{i \in \sigma} \|f_i\|^2 + \sum\limits_{i \in \sigma^c} \|g_i\|^2 \right) \leq \left\|\sum\limits_{i \in \sigma} v_i f_i + \sum\limits_{i \in \sigma^c} w_i g_i \right\|^2 \leq B_{\mathcal V \mathcal W} \left (\sum\limits_{i \in \sigma} \|f_i\|^2 + \sum\limits_{i \in \sigma^c} \|g_i\|^2 \right)
\end{equation}
Furthermore, $\{\mathcal V_i\}_{i \in \sigma} \cup \{\mathcal W_i\}_{i \in \sigma^c} $ is called an orthonormal weaving fusion basis in $\mathcal H$ if $\mathcal I_{\mathcal H}=\sum\limits_{i \in \sigma} P_{\mathcal V_i} + \sum\limits_{i \in \sigma^c}  P_{\mathcal W_i}$ and $\mathcal V_i \perp \mathcal V_j$, $\mathcal W_i \perp \mathcal W_j$ for every $i \neq j$.

We characterize weaving fusion Riesz bases using weaving fusion frame synthesis and analysis operators, establishing their structural properties. This characterization leads to the fact that every weaving fusion Riesz basis is also a weaving fusion frame, underscoring their dual role in the frame theory.

\begin{thm}
$\mathcal V_v=\{(\mathcal V_i, v_i)\}_{i \in I}$ and $\mathcal W_w=\{(\mathcal W_i, w_i)\}_{i \in I}$ be two weighted sequences of closed subspaces of $\mathcal H$. Then the following are equivalent:
\begin{enumerate}
\item $\mathcal V_v$ and $\mathcal W_w$ are weaving fusion Riesz bases in $\mathcal H$.

\item For every $\sigma \subset I$, the corresponding weaving synthesis operator $T_{\mathcal V \mathcal W}: \mathcal L_{\mathcal V \mathcal W}^2=\left(\sum\limits_{i \in \sigma} \bigoplus \mathcal V_i + \sum\limits_{i \in \sigma^c} \bigoplus \mathcal W_i\right)_{l^2} \rightarrow \mathcal H$ is bounded with $R(T_{\mathcal V \mathcal W})=\mathcal H$ and $N(T_{\mathcal V \mathcal W})=(\mathcal L_{\mathcal V \mathcal W}^2)^{\perp}$.

\item For every $\sigma \subset I$, the associated weaving analysis operator $T_{\mathcal V \mathcal W}^*: \mathcal H \rightarrow \mathcal L_{\mathcal V \mathcal W}^2$ is bounded with $R(T_{\mathcal V \mathcal W}^*)=\mathcal L_{\mathcal V \mathcal W}^2$ and $N(T_{\mathcal V \mathcal W}^*)=\{0\}$.
\end{enumerate}
\end{thm}

\begin{proof}
Applying Lemma \ref{Lemma}, for every $\sigma \subset I$ the orthogonal complement $(\mathcal L_{\mathcal V \mathcal W}^2)^{\perp}$ in $\mathcal L_{\mathcal H}^2$ is given by $\left(\sum\limits_{i \in \sigma} \bigoplus \mathcal V_i^{\perp} + \sum\limits_{i \in \sigma^c} \bigoplus \mathcal W_i ^{\perp}\right)_{l^2}$. Moreover, we have $T_{\mathcal V \mathcal W} = T_{\mathcal V \mathcal W} P_{\mathcal L_{\mathcal V \mathcal W}^2}$. Thus the condition in $(2)$ is equivalent to the fact that $T_{\mathcal V \mathcal W | \mathcal L_{\mathcal V \mathcal W}^2}$ is bounded, bijective operator and  the condition in $(3)$ is equivalent to the fact that $T_{\mathcal V \mathcal W}^*$ is also bounded,  bijective operator. 

Furthermore, since $R(T_{\mathcal V \mathcal W}^*) \subset \mathcal L_{\mathcal V \mathcal W}^2$, then we have $(T_{\mathcal V \mathcal W | \mathcal L_{\mathcal V \mathcal W}^2})^* = T_{\mathcal V \mathcal W}^*$.\\

\noindent (\underline{1 $\implies$ 2}) Let $\mathcal V_v$ and $\mathcal W_w$ are weaving fusion Riesz bases in $\mathcal H$. Then from the right inequality of the inequality \ref{fusion Riesz} we have, 
$T_{\mathcal V \mathcal W}$ and $T_{\mathcal V \mathcal W}^*$ are bounded.

If possible $T_{\mathcal V \mathcal W | \mathcal L_{\mathcal V \mathcal W}^2}$ is not surjective, then there exists $0 \neq f \in R(T_{\mathcal V \mathcal W | \mathcal L_{\mathcal V \mathcal W}^2})^{\perp}$ with $f \perp (\mathcal V_{i\in \sigma} \cup \mathcal W_{i \in \sigma^c})$ for every $\sigma \subset I$. Again since $\{\mathcal V_i\}_{i \in \sigma} \cup \{\mathcal W_i\}_{i \in \sigma^c}$ is complete, then there exists a sequence $\{h_n\}_{n=1}^\infty \in \text{span} (\{\mathcal V_i\}_{i \in \sigma} \cup \{\mathcal W_i\}_{i \in \sigma^c})$ so that we have,
\begin{eqnarray*}
0=\lim\limits_{n \rightarrow \infty} \|f - h_n\| = \lim\limits_{n \rightarrow \infty} (\|f\|^2 - \langle f, h_n \rangle - \langle h_n, f  \rangle +\|h_n\|^2)&=& 2\|f\|^2 \\
&>&0,
\end{eqnarray*}
which is a contradiction. Thus $T_{\mathcal V \mathcal W | \mathcal L_{\mathcal V \mathcal W}^2}$ is surjective.

Applying Theorem \ref{1}, for every $\sigma \subset I$, $\{(\mathcal V_i, v_i)\}_{i \in \sigma} \cup \{(\mathcal W_i, w_i)\}_{i \in \sigma^c}$ is a fusion frame and hence $\left (\sum\limits_{i \in \sigma}v_i f_i + \sum\limits_{i \in \sigma^c}w_i g_i \right )$ converges unconditionally, where $\{f_i\}_{i \in \sigma} \cup \{g_i\}_{i \in \sigma^c} \in  \mathcal L_{\mathcal V \mathcal W}^2$.

Furthermore, from the left inequality of the inequality \ref{fusion Riesz} it is easy to see that $T_{\mathcal V \mathcal W | \mathcal L_{\mathcal V \mathcal W}^{00}}$ is injective. Since $T_{\mathcal V \mathcal W}$ is bounded, $\mathcal L_{\mathcal V \mathcal W}^{00}$ is dense in $\mathcal L_{\mathcal V \mathcal W}^2$ and $\left (\sum\limits_{i \in \sigma}v_i f_i + \sum\limits_{i \in \sigma^c}w_i g_i \right )$ is unconditionally convergent, then the left inequality of the inequality \ref{fusion Riesz} will hold for every $\{f_i\}_{i \in \sigma} \cup \{g_i\}_{i \in \sigma^c} \in  \mathcal L_{\mathcal V \mathcal W}^2$, for every $\sigma \subset I$. Therefore, $T_{\mathcal V \mathcal W | \mathcal L_{\mathcal V \mathcal W}^2}$ is injective and hence $T_{\mathcal V \mathcal W | \mathcal L_{\mathcal V \mathcal W}^2}$ is bijective.\\

\noindent (\underline{2 $\implies$ 1}) Let  $T_{\mathcal V \mathcal W | \mathcal L_{\mathcal V \mathcal W}^2}$ is bijective operator. Then the inequality \ref{fusion Riesz} will hold for every $\{f_i\}_{i \in \sigma} \cup \{g_i\}_{i \in \sigma^c} \in  \mathcal L_{\mathcal V \mathcal W}^2$, for every $\sigma \subset I$. 

If possible for every $\sigma \subset I$, $\{\mathcal V_i\}_{i \in \sigma} \cup \{\mathcal W_i\}_{i \in \sigma^c}$ is not complete, then there exists $0 \neq f \in \mathcal H$ so that $f \perp \overline{\text{span}} (\{\mathcal V_i\}_{i \in \sigma} \cup \{\mathcal W_i\}_{i \in \sigma^c})$. Therefore, $T_{\mathcal V \mathcal W}^*f=0$, which is a contradiction. Hence $\mathcal V_v$ and $\mathcal W_w$ are weaving fusion Riesz bases.

Finally, applying Theorem \ref{1}, we have every weaving fusion Riesz basis is also weaving fusion frame. Therefore, $T_{\mathcal V \mathcal W} T_{\mathcal V \mathcal W}^*$ and $T_{\mathcal V \mathcal W}^* T_{\mathcal V \mathcal W}$ have the same non-zero spectrum. Thus weaving fusion frames and weaving fusion Riesz bases have same bounds.\\

\noindent (\underline{2 $\Longleftrightarrow$ 3}) This will hold for a bounded bijective operator and its Hilbert adjoint (see \cite{Co90}).
\end{proof}




\end{document}